\documentclass[12pt]{amsart}
\usepackage{amsmath}
\usepackage{amssymb}
\usepackage[all]{xy}
\usepackage{longtable}
\usepackage{float}
\usepackage{placeins}
\usepackage{color}
\usepackage{rotating,booktabs}
\allowdisplaybreaks
\newtheorem{thm}[equation]{Theorem}
\newtheorem{lem}[equation]{Lemma}
\newtheorem{prop}[equation]{Proposition}
\newtheorem{cor}[equation]{Corollary}
\newtheorem{conj}[equation]{Conjecture}
\newtheorem{rem}[equation]{Remark}

\theoremstyle{definition}

\theoremstyle{remark}

\numberwithin{equation}{section}

\DeclareMathAlphabet{\matheur}{U}{eur}{m}{n}

\newcommand{\NN}{\mathbb{N}}
\newcommand{\ZZ}{\mathbb{Z}}
\newcommand{\QQ}{\mathbb{Q}}

\DeclareMathOperator{\Gal}{Gal}

\DeclareMathOperator{\im}{Im}
\DeclareMathOperator{\re}{Re}

\DeclareMathOperator{\Li}{Li}
\DeclareMathOperator{\Reg}{Reg}
\DeclareMathOperator{\Sym}{Sym}

\newmuskip\pFqskip
\mathchardef\pFcomma=\mathcode`, 

\begin{document}

\title[On a non-critical symmetric square $L$-value]{On a non-critical symmetric square $L$-value of the congruent number elliptic curves}

\author{Detchat Samart}
\address{Department of Mathematics, Burapha University, Chonburi, 20131, Thailand} \email{petesamart@gmail.com}



\date{\today}

\begin{abstract}
The congruent number elliptic curves are defined by $E_d: y^2=x^3-d^2x$, where $d\in \NN.$ We give a simple proof of a formula for $L(\Sym^2(E_d),3)$ in terms of the determinant of the elliptic trilogarithm evaluated at some degree zero divisors supported on the torsion points on $E_d(\overline{\QQ})$. 
\end{abstract}


\maketitle
\section{Introduction}\label{Sec:Intro}
Let $E$ be an elliptic curve defined over $\mathbb{C}$. Then there exist $\tau\in\mathbb{C}$ such that $\im(\tau)>0$ and the following isomorphisms:
\begin{equation}\label{E:isom}
\begin{aligned}
\qquad E(\mathbb{C})\qquad& \tilde{\longrightarrow} & \mathbb{C}/\Lambda\qquad & \tilde{\longrightarrow} & \mathbb{C}^{\times}/q^{\mathbb{Z}}\\
\left(\wp_{\Lambda}(u),\wp_{\Lambda}'(u)\right)\qquad & \longmapsto &  u \pmod \Lambda \qquad & \longmapsto & e^{2\pi i u},
\end{aligned}
\end{equation}
where $\Lambda=\mathbb{Z}+\mathbb{Z}\tau$, and $\wp_{\Lambda}$ is the Weierstrass $\wp$-function. 
Zagier and Gangl \cite[\S 10]{ZG} defined the two functions $\mathcal{L}^E_{3,j}:E(\mathbb{C})\rightarrow\mathbb{R}, j=1,2,$ by 
\begin{align*}
\mathcal{L}^E_{3,1}(P):&=\mathcal{L}_{3,1}^E(x)=\sum_{n=-\infty}^{\infty}\mathcal{L}_3(q^n x),\\
\mathcal{L}^E_{3,2}(P):&=\mathcal{L}^E_{3,2}(x)=\sum_{n=0}^{\infty}J_3(q^n x)+\sum_{n=1}^{\infty}J_3(q^n x^{-1})+\frac{\log^2|x|\log^2|qx^{-1}|}{4\log|q|},
\end{align*}
where $\mathcal{L}_3(z)=\re\left(\Li_3(z)-\log|z|\Li_2(z)+\frac{1}{3}\log^2|z|\Li_1(z)\right),$ $\Li_m(z)=\sum_{n=1}^{\infty}\frac{z^n}{n^m},$ the classical $m^\text{th}$ polylogarithm function, $J_3(x)=\log^2|x|\log|1-x|,$ $q=e^{2\pi i \tau},$ and $x\in\mathbb{C}^{\times}$ is the image of $P$ under the composition of the isomorphisms above. The function $\mathcal{L}^E_{3,1}$ is called the \textit{elliptic trilogarithm}. These two functions serve as higher-dimensional analogues of the \textit{elliptic dilogarithm} $D^E$ and the function $J^{E}$ defined by
\begin{align*}
D^E(x)&=\sum_{n=-\infty}^{\infty}\mathcal{L}_2(q^n x),\\
J^E(x)&=\sum_{n=0}^{\infty}J(q^n x)-\sum_{n=1}^{\infty}J(q^n x^{-1})+\frac{1}{3}\log^2|q|B_3\left(\frac{\log|x|}{\log|q|}\right),
\end{align*}
where $\mathcal{L}_2(z)=\im(\Li_2(z)+\log |z|\log(1-z))$, known as the \textit{Bloch-Wigner dilogarithm}, $J(z)=\log|z|\log|1-z|$, and $B_3(X)=X^3-3X^2/2+X/2$.

Recall that for $a,b\in\mathbb{N}$ the series
\begin{equation*}
K_{a,b}(\tau;u)=\sideset{}{'}\sum_{m,n\in\mathbb{Z}}\frac{e^{2\pi i(n\xi-m\eta)}}{(m\tau+n)^a(m\bar{\tau}+n)^b},
\end{equation*}
where $u=\xi\tau+\eta$ and $\xi,\eta\in\mathbb{R}/\mathbb{Z}$, is called the \textit{Eisenstein-Kronecker series.}  Here and throughout, $\sideset{}{'}\sum$ means  $(m,n)\neq (0,0)$ in the summation. Bloch \cite{Bloch} defined the regulator function $R^E:E(\mathbb{C})\rightarrow \mathbb{R}$ by
\begin{equation*}
R^E(e^{2\pi i u})=\frac{\im(\tau)^2}{\pi}K_{2,1}(\tau;u).
\end{equation*} 
One can extend the functions $D^E,J^E,R^E,\mathcal{L}^E_{3,1}$ and $\mathcal{L}^E_{3,2}$ to the group of divisors on $E(\mathbb{C})$ by linearity. Also, it can be shown that $\re(R^E)=D^E$ and $\im(R^E)=J^E.$ In \cite{GL}, Goncharov and Levin prove the following theorem, formerly known as Zagier's conjecture on $L(E,2)$:
\begin{thm}\label{T:G}
Let $E$ be a modular elliptic curve over $\mathbb{Q}$. Then there exists a divisor $P=\sum n_j(P_j)$ on $E(\bar{\mathbb{Q}})$ satisfying the following conditions:
\begin{itemize}
\item [a)] $$\sum n_j P_j\otimes P_j\otimes P_j =0 \text{ in } \Sym^3(E),$$
\item [b)] For any valuation $v$ of the field $\mathbb{Q}(P)$ generated by the coordinates of the points $P_j$
\begin{equation*}
\sum n_j h_v(P_j)\cdot P_j =0 \text{ on } E,
\end{equation*}
where $h_v$ is the local height associated to the valuation $v$.
\item [c)] For every prime $p$ where $E$ has a split multiplicative reduction, $P$ satisfies a \textit{certain} integrality condition. (see \cite[Thm.~1.1]{GL})
\end{itemize}
Moreover, for such a divisor $P$, 
\begin{equation*}
L(E,2) \sim_{\mathbb{Q}^\times} \pi\cdot D^E(P),
\end{equation*}
where $A\sim_{\mathbb{Q}^\times}B$ means $A=cB$ for some $c\in\mathbb{Q}^\times.$
\end{thm}
There are several numerical results and conjectures relating special values of $L$-series of symmetric powers of an elliptic curve over $\mathbb{Q}$ to higher elliptic polylogarithms including those due to Mestre and Schappacher \cite{MS}, Goncharov \cite{Goncharov}, and Wildeshaus \cite{Wildeshaus}. Inspired by these examples and their numerical experiment, Zagier and Gangl \cite[\S 10]{ZG} formulated the following conjecture, which is an analogue of Theorem~\ref{T:G}:
\begin{conj}\label{C:Zag}
Let $E$ be an elliptic curve over $\mathbb{Q}$. For any $\xi=\sum n_i(P_i) \in \mathbb{Z}[E(\bar{\mathbb{Q}})]$ and any homomorphism $\phi:E(\bar{\mathbb{Q}})\rightarrow \mathbb{Z}$, let $\iota_\phi(\xi)=\sum n_i\phi(P_i)(P_i)$. Also define $$\mathcal{C}_2(E/\mathbb{Q})=\left\langle (f)\diamond(1-f), (P)+(-P), (2P)-2\sum_{T\in E[2]}(P+T) \mid f\in\mathbb{Q}(E), P\in E(\mathbb{Q}) \right\rangle$$ as a subgroup of $\mathbb{Z}[E(\bar{\mathbb{Q}})]^{\Gal(\bar{\mathbb{Q}}/\mathbb{Q})},$ where for $(f)=\sum m_i(P_i)$ and $(g)=\sum n_j(Q_j)$ the diamond operator $\diamond$ is defined by $(f)\diamond (g)=\sum m_i n_j(a_i-b_j).$ If $\iota_\phi(\xi)\in\mathcal{C}_2(E/\mathbb{Q})$ for all homomorphisms $\phi:E(\bar{\mathbb{Q}})\rightarrow \mathbb{Q}$, then $\vec{\mathcal{L}}_3^E(\xi):=(\mathcal{L}_{3,1}^E(\xi),\mathcal{L}_{3,2}^E(\xi))$ belongs to a $2$-dimensional lattice whose covolume is related to $L(\Sym^2(E),3).$
\end{conj}
They also verified numerically that if $E$ is the conductor $37$ elliptic curve defined by $y^2-y=x^3-x$, $E\cong \mathbb{C}/(\mathbb{Z}+\mathbb{Z}\tau)$, and 
\begin{align*}
\eta_4&=3(4P)-13(3P)+18(2P)-3(P)-5(O),\\
\eta_6&=2(6P)-45(3P)+60(2P)+93(P)-110(O),
\end{align*}
where $P=[0,0]$, then 
\begin{equation}\label{E:Zag}
\Reg_3(E):=\begin{vmatrix}
\mathcal{L}^E_{3,1}(\eta_4) & \mathcal{L}^E_{3,2}(\eta_4) \\ 
\mathcal{L}^E_{3,1}(\eta_6) & \mathcal{L}^E_{3,2}(\eta_6)
\end{vmatrix}\stackrel{?}=-\frac{37^3}{4}\im(\tau)^2 L(\Sym^2(E),3).
\end{equation}
(Note that the negative sign in the above identity is missing in \cite{ZG}.) Recall from \cite{CS} that $L(\Sym^2(E),s)$ satisfies the functional equation
\begin{equation*}
\Lambda(\Sym^2(E),s)=\Lambda(\Sym^2(E),3-s),
\end{equation*}
where $\Lambda(\Sym^2(E),s)=C^{s/2}\pi^{-s/2}\Gamma\left(\frac{s}{2}\right)(2\pi)^{-s}\Gamma(s)L(\Sym^2(E),s)$, and $C$ is the conductor of the Galois representation associated to the symmetric square of the Tate module of $E$.
Therefore, \eqref{E:Zag} can be rephrased as 
\begin{equation*}
\Reg_3(E) \stackrel{?}= 2\pi^4 \im(\tau)^2 L''(\Sym^2(E),0).
\end{equation*}
This conjecture is consistent with a special case of \cite[Conj.6.8]{Goncharov}, namely, for any elliptic curve $E$ over $\mathbb{Q}$ there exist degree zero divisors $\xi_1$ and $\xi_2$ on $E(\bar{\mathbb{Q}})$ such that
\begin{equation*}
\begin{vmatrix}
\re\left(K_{1,3}(\tau;\xi_1)\right) & K_{2,2}(\tau;\xi_1) \\ 
\re\left(K_{1,3}(\tau;\xi_2)\right) & K_{2,2}(\tau;\xi_2)
\end{vmatrix}\stackrel{?}\sim_{\mathbb{Q}^\times}\frac{\pi^6}{\im(\tau)^4}L''(\Sym^2(E),0).
\end{equation*}
The relationship between the determinant above and the one in \eqref{E:Zag} was established in \cite[\S 4]{Samart} and can be stated as follows:
\begin{prop}
Let $E$ be an elliptic curve over $\mathbb{C}$ and suppose that $E\cong \mathbb{C}/(\mathbb{Z}+\mathbb{Z}\tau)$. If $\xi_1$ and $\xi_2$ are divisors of degree zero on $E$, then 
\begin{equation*}
\begin{vmatrix}
\mathcal{L}^E_{3,1}(\xi_1) & \mathcal{L}^E_{3,2}(\xi_1) \\ 
\mathcal{L}^E_{3,1}(\xi_2) & \mathcal{L}^E_{3,2}(\xi_2)
\end{vmatrix}=-\frac{2\im(\tau)^6}{\pi^2}\begin{vmatrix}
\re\left(K_{1,3}(\tau;\xi_1)\right) & K_{2,2}(\tau;\xi_1) \\ 
\re\left(K_{1,3}(\tau;\xi_2)\right) & K_{2,2}(\tau;\xi_2)
\end{vmatrix},
\end{equation*}
where $K_{a,b}(\tau;\xi)=\sum_{P\in E}n_P K_{a,b}(\tau;u_P)$ if $\xi=\sum_{P\in E}n_P (P)$ and $u_P$ is the image of $P$ in $\mathbb{C}/(\mathbb{Z}+\mathbb{Z}\tau)$.
\end{prop}
The main result in this paper concerns a symmetric square $L$-value of the congruent number elliptic curves, which are defined by
\begin{equation}\label{E:CNEC}
E_d: y^2=x^3-d^2x, \quad d\in\NN.
\end{equation}
These curves play a crucial role in the study of the congruent number problem\footnote{A square-free positive integer $n$ is called a \textit{congruent number} if it is the area of a right triangle whose all sides are rational numbers. The congruent number problem asks if  there is an algorithm for determining whether any given number is congruent in a finite number of steps.}, which is one of the oldest unsolved problems in number theory. More precisely, assuming the Birch and Swinnerton-Dyer conjecture, it can be proven that a square-free positive integer $d$ is congruent if and only if $L(E_d,1)=0$ (see, for example, \cite{Koblitz}). Some useful facts about $E_d$ include $E_d\cong \mathbb{C}/(\mathbb{Z}+\mathbb{Z}\sqrt{-1})$, that $E_d$ has complex multiplication by $\ZZ[\sqrt{-1}]$ and that $E_d$ is a quadratic twist of $E_1$. We will give a rigorous proof of a formula for $L(\Sym^2(E_d),3)$, which provides an evidence supporting the aforementioned conjecture. 
\begin{thm}\label{T:main}
For any positive integer $d$, let $E:=E_d$ be the elliptic curve defined by \eqref{E:CNEC} and let $P,Q,$ and $O$ be points on $E(\bar{\mathbb{Q}})$ corresponding to $\sqrt{-1}/2,1/4,$ and $1$, respectively, via the isomorphism $E\cong \mathbb{C}/(\mathbb{Z}+\mathbb{Z}\sqrt{-1}).$ If $\xi_1=(Q)+(P+Q)-2(O)$ and $\xi_2=(2Q)-(P)$, then the following identity is true:
\begin{equation}\label{E:main}
\begin{vmatrix}
\mathcal{L}^E_{3,1}(\xi_1) & \mathcal{L}^E_{3,2}(\xi_1) \\ 
\mathcal{L}^E_{3,1}(\xi_2) & \mathcal{L}^E_{3,2}(\xi_2)
\end{vmatrix}=-\frac{43}{2}L(\Sym^2(E),3)=-\frac{43\pi^4}{128}L''(\Sym^2(E),0).
\end{equation}
\end{thm}

\begin{rem}
\begin{itemize}
\item[(i)] The points $P$ and $Q$ in Theorem~\ref{T:main} can be written explicitly as $P=[d,0]$ and $Q=\left[-d(1+\sqrt{2}),\sqrt{-(6+4\sqrt{2})d^3}\right]$.

\item[(ii)] Since the symmetric square $L$-function is invariant under a quadratic twist, it suffices to prove Theorem~\ref{T:main} for a particular value of $d$. As the reader will see in Section~\ref{Sec:Gross} and Section~\ref{Sec:proof}, we choose $d=2$.
\end{itemize}
\end{rem}

\section{Some identities involving $\mathcal{L}^E_{3,1}$ and $\mathcal{L}^E_{3,2}$}\label{Sec:LE}

Before proving the main result, we shall state some useful facts about the functions $\mathcal{L}^E_{3,1}$ and $\mathcal{L}^E_{3,2}$. The reader is referred to \cite{Samart} and \cite{Zagier} for further details.

\begin{prop}[{\cite[Cor.~2.3]{Samart}}]
Suppose that $E\cong \mathbb{C}/(\mathbb{Z}+\mathbb{Z}\tau)$ with $\tau\in\mathcal{H}$ and let $q=e^{2\pi i\tau}$ and $x=e^{2\pi i u}$, where $u=\xi\tau+\eta$ and $\xi,\eta\in\mathbb{R}/\mathbb{Z}$. Then the following identities hold:

\begin{align}
\mathcal{L}_{3,1}^E(x)&=\frac{4\im(\tau)^5}{3\pi}\re\left(\sideset{}{'}\sum_{m,n\in\mathbb{Z}}e^{2\pi i(n\xi-m\eta)}\frac{m^2}{|m\tau+n|^6}\right),\label{E:L31}\\
\mathcal{L}_{3,2}^E(x)&=\frac{\im(\tau)^3}{\pi}\left[\sideset{}{'}\sum_{m,n\in\mathbb{Z}}\frac{e^{2\pi i(n\xi-m\eta)}}{|m\tau+n|^4}+2\re\left(\sideset{}{'}\sum_{m,n\in\mathbb{Z}}e^{2\pi i(n\xi-m\eta)}\frac{(m\tau+n)^2}{|m\tau+n|^6}\right)\right]+\frac{\log^3|q|}{120} \label{E:L32}.
\end{align}
\end{prop}

The following result is an immediate consequence of \eqref{E:L31}.

\begin{prop}
Let $E$ be an elliptic curve isomorphic to $\mathbb{C}/(\mathbb{Z}+\mathbb{Z}\tau).$ If $P$ and $Q$ are the points on $E$ corresponding to $\tau/2$ and $1/4$, respectively, via the isomorphism above, then
\begin{equation}
\mathcal{L}^E_{3,1}((Q)+(P+Q))=\frac{1}{8}\mathcal{L}^E_{3,1}(2Q). \label{E:linear}
\end{equation} 
\end{prop}
\begin{proof}
Using \eqref{E:L31} and the fact that $e^{\frac{m}{2}\pi i}=i^m$ for any $m\in \mathbb{Z}$, we have 
\begin{align*}
\mathcal{L}^E_{3,1}(8(Q)+8(P+Q))&=\frac{32\im(\tau)^5}{3\pi}\re\left(\sideset{}{'}\sum_{m,n\in\mathbb{Z}}\left(i^m\left(1+(-1)^n\right)\right)\frac{m^2}{|m\tau+n|^6}\right)\\
&=\frac{128\im(\tau)^5}{3\pi}\sideset{}{'}\sum_{m,n\in\mathbb{Z}}\left((-1)^m\left(1+(-1)^n\right)\right)\frac{m^2}{|2m\tau+n|^6}\\
&=\frac{4\im(\tau)^5}{3\pi}\sideset{}{'}\sum_{m,n\in\mathbb{Z}}\frac{(-1)^m m^2}{|m\tau+n|^6}\\
&=\mathcal{L}^E_{3,1}(2Q).
\end{align*}
The last equality follows from the fact that $2Q$ is a point corresponding to $1/2$.
\end{proof}

\section{Gr\"{o}ssencharakters and modular forms}\label{Sec:Gross}

It is a well-known fact that the $L$-function of an elliptic curve over $\mathbb{Q}$ with complex multiplication coincides with that of a Hecke character (a.k.a. a Gr\"{o}ssencharakter)  of an imaginary quadratic field. In this section, we will explicitly construct the Hecke character corresponding to the CM elliptic curve $E:=E_2$. Then we invoke a result of Coates and Schmidt \cite{CS} to obtain an expression of $L(\Sym^2(E),s)$ in terms of a product of $L$-functions attached to a weight $3$ modular form and a Dirichlet character. More precisely, we will prove the following identity:

\begin{thm}\label{T:CS}
Let $E$ be the conductor $64$ defined by $E: y^2=x^3-4x$. Then for any $s\in \mathbb{C}$, we have
$$L(\Sym^2(E),s)=L(g,s)L(\chi_{-4},s-1),$$
where $g(\tau)=q-6q^5+9q^9+\cdots $ is a weight $3$ cusp form of level $16$ and $\chi_{-4}=\left(\frac{-4}{\cdot}\right),$ the Dirichlet character associated to $\mathbb{Q}(\sqrt{-1})$.
\end{thm}
\begin{proof}
Let $K=\mathbb{Q}(\sqrt{-1})$. Then the ring of integer of $K$ is $\mathcal{O}_K=\mathbb{Z}[\sqrt{-1}].$ Let $\Lambda = (4)\subset \mathcal{O}_K$ and let $P(\Lambda)$ be the set of (integral) ideals of $\mathcal{O}_K$ which are relatively prime to $\Lambda.$ Then it is easily seen that each element of $P(\Lambda)$ can be represented uniquely by $(m+ni)$, where $m>0$ is an odd integer and $n$ is an even integer. 
 
Define a map $\phi:P(\Lambda)\rightarrow \mathbb{C}^\times$ by 
\begin{equation*}
\phi((m+ni))= \chi_{-4}(m)(m+ni)=
  \begin{cases}
    m+ni, & \text{if } m\equiv 1 \pmod 4, \\
   -(m+ni), & \text{if }  m\equiv 3 \pmod 4.
  \end{cases}
\end{equation*}
Hence for any $\alpha\in \mathcal{O}_K$ such that $\alpha \equiv 1 \pmod \Lambda$ we have $\phi((\alpha))= \alpha.$ It follows that we can extend $\phi$ multiplicatively to a Hecke character of conductor $\Lambda$. Thus, by \cite[Thm.~1.31]{Ono}, 
\[\mathfrak{f}(\tau)= \sum_{\mathfrak{a}\in P(\Lambda)}\phi(\mathfrak{a}) q^{N(\mathfrak{a})}\]
is a weight $2$ newform of level $64$. Computing the first few terms of $\mathfrak{a}$, we obtain
\[\mathfrak{f}(\tau)=\sum_{\substack{m \in \mathbb{N}\\n\in\mathbb{Z} }}\chi_{-4}(m)m q^{m^2+4n^2}= q+2q^5-3q^9-6q^{13}+\cdots,\]
which is the weight $2$ newform corresponding to $E$ via the modularity theorem.

Let $\phi^2$ be the primitive Hecke character attached to the square of $\phi$. Then $\phi^2$ is a Hecke character of conductor $\Lambda' = (2)$ and satisfies 
\[\phi^2((\alpha))=\alpha^2,\]
for any ideal $(\alpha)$ in $\mathcal{O}_K$ satisfying $\alpha \equiv 1 \pmod{\Lambda'}.$ Moreover, it is known that $L(\phi^2,s) = L(g,s)$ (see, for example, \cite[Lem.~2.3]{Samart1}). Finally, by a result due to Coates and Schmidt \cite[Prop~5.1]{CS}, we have 
\[L(\Sym^2(E),s)=L(\phi^2,s)L(\chi_{-4},s-1)= L(g,s)L(\chi_{-4},s-1).\]
\end{proof}

It has been shown that $L(\chi_{-4},s)$ and $L(g,s)$ have simple lattice sum expressions, which will be particularly useful in the proof of our main result. Let us finish this section by stating these results.

\begin{prop}[{\cite[Lem.~2.3]{Samart1}},{\cite[Sect.~IV]{GZ}}]
Let $s,t\in \mathbb{C}$, where $\re(s)>2$ and $\re(t)>1$. Then we have
\begin{align*}
L(g,s) &= \frac{1}{2} \sideset{}{'}\sum_{m,n\in\mathbb{Z}}\frac{m^2-4n^2}{(m^2+4n^2)^s},\\
L(\chi_{-4},t) &=\frac{1}{4\zeta(t)}\sideset{}{'}\sum_{m,n\in\mathbb{Z}} \frac{1}{(m^2+n^2)^t},\\
&=\frac{1}{2(1-2^{-t}+2^{1-2t})\zeta(t)}\sideset{}{'}\sum_{m,n\in\mathbb{Z}} \frac{1}{(m^2+4n^2)^t},
\end{align*}
where $\zeta(t)$ is the Riemann zeta function.

\end{prop}

\begin{cor}
The following formulas are true:
\begin{align}
L(\chi_{-4},2)&= \frac{3}{2\pi^2} \sideset{}{'}\sum_{m,n\in\mathbb{Z}} \frac{1}{(m^2+n^2)^2},\label{E:ch1}\\
&=\frac{24}{7\pi^2} \sideset{}{'}\sum_{\substack{m \text{ even}\\ n\in \mathbb{Z}}} \frac{1}{(m^2+n^2)^2},\label{E:ch2}\\
L(g,3) &= \frac{1}{2} \sideset{}{'}\sum_{\substack{m \in \mathbb{Z}\\n \text{ even}}}\frac{m^2-n^2}{(m^2+n^2)^3}, \label{E:m1}\\
&=\frac{1}{2}\sum_{\substack{m \text{ odd}\\n \text{ even}}}\frac{m^2-n^2}{(m^2+n^2)^3}. \label{E:m2}
\end{align}
\end{cor}

\section{Proof of the main result}\label{Sec:proof}
We precede the proof of \eqref{E:main} by a lemma consisting of a series of identities relating values of $\mathcal{L}^E_{3,1}$ and $\mathcal{L}^E_{3,2}$ to modular and Dirichlet $L$-values. 
\begin{lem}\label{L:1}
With the same assumption in Theorem~\ref{T:main}, the following identities are true:
\begin{align}
\mathcal{L}^E_{3,1}((Q)+(P+Q))&= -\frac{1}{3\pi} L(g,3)-\frac{\pi}{144}L(\chi_{-4},2),\label{L1}\\
\mathcal{L}^E_{3,1}(O) &= \frac{4\pi}{9}L(\chi_{-4},2), \label{L2}\\
\mathcal{L}^E_{3,1}((2Q)-(P)) &= -\frac{16}{3\pi}L(g,3), \label{L3}\\
\mathcal{L}^E_{3,2}((2Q)-(P)) &= \frac{16}{\pi}L(g,3), \label{L4}\\
\mathcal{L}^E_{3,2}((Q)+(P+Q)-2(O)) &= \frac{1}{\pi}L(g,3)-\frac{43\pi}{32}L(\chi_{-4},2). \label{L5}
\end{align}
\end{lem}

\begin{proof}
Note first that, by symmetry, 

\begin{equation*}
\sideset{}{'}\sum_{m,n\in\mathbb{Z}} \frac{n^2-m^2}{(m^2+n^2)^3}=0.
\end{equation*}
Therefore, by \eqref{E:L31} and \eqref{E:ch1}, we have 
\begin{align*}
\mathcal{L}^E_{3,1}(O) &= \frac{4}{3\pi}\re\left(\sideset{}{'}\sum_{m,n\in\mathbb{Z}}e^{-2\pi im}\frac{m^2}{|n+mi|^6}\right)\\
&= \frac{4}{3\pi}\sideset{}{'}\sum_{m,n\in\mathbb{Z}}\frac{m^2}{(m^2+n^2)^3}\\
&= \frac{2}{3\pi}\sideset{}{'}\sum_{m,n\in\mathbb{Z}}\left(\frac{1}{(m^2+n^2)^2}-\frac{n^2-m^2}{(m^2+n^2)^3}\right)\\
&= \frac{2}{3\pi}\sideset{}{'}\sum_{m,n\in\mathbb{Z}}\frac{1}{(m^2+n^2)^2}\\
&= \frac{4\pi}{9}L(\chi_{-4},2),
\end{align*}
which yields \eqref{L2}.

Next, using \eqref{E:L32} and \eqref{E:m2}, we obtain
\begin{align*}
\mathcal{L}^E_{3,2}((2Q)-(P)) &= \frac{1}{\pi}\left(\sideset{}{'}\sum_{m,n\in\mathbb{Z}}\frac{(-1)^m-(-1)^n}{(m^2+n^2)^2}+2\sideset{}{'}\sum_{m,n\in\mathbb{Z}}((-1)^m-(-1)^n)\frac{n^2-m^2}{(m^2+n^2)^3}\right)\\
&= \frac{1}{\pi}\Bigg(-2\sideset{}{'}\sum_{\substack{m \text{ odd}\\ n\text{ even}}}\frac{1}{(m^2+n^2)^2}-4\sideset{}{'}\sum_{\substack{m \text{ odd}\\ n\text{ even}}}\frac{n^2-m^2}{(m^2+n^2)^3}+2\sideset{}{'}\sum_{\substack{m \text{ even}\\ n\text{ odd}}}\frac{1}{(m^2+n^2)^2}\\
&\qquad\qquad +4\sideset{}{'}\sum_{\substack{m \text{ even}\\ n\text{ odd}}}\frac{n^2-m^2}{(m^2+n^2)^3} \Bigg)\\
&= \frac{8}{\pi}\sideset{}{'}\sum_{\substack{m \text{ even}\\ n\text{ odd}}}\frac{n^2-m^2}{(m^2+n^2)^3}\\
&= \frac{16}{\pi}L(g,3),
\end{align*} 
which is \eqref{L4}. \\

On the other hand, it is easily seen by symmetry that $\displaystyle\sideset{}{'}\sum_{m,n\in\mathbb{Z}}\frac{(-1)^m-(-1)^n}{(m^2+n^2)^2}$ vanishes, so we have 
\begin{align*}
\mathcal{L}^E_{3,2}((2Q)-(P)) &=\frac{2}{\pi}\sideset{}{'}\sum_{m,n\in\mathbb{Z}}((-1)^m-(-1)^n)\frac{n^2-m^2}{(m^2+n^2)^3}\\
&= -\frac{4}{\pi}\sideset{}{'}\sum_{m,n\in\mathbb{Z}}((-1)^m-(-1)^n)\frac{m^2}{(m^2+n^2)^3}\\
&= -3\mathcal{L}^E_{3,1}((2Q)-(P)).
\end{align*}
Together with \eqref{L4}, this implies \eqref{L3}. 

To establish \eqref{L1}, we first employ \eqref{E:ch1} and \eqref{E:ch2} to deduce that
\begin{equation*}
\frac{1}{2}\sideset{}{'}\sum_{m,n\in\mathbb{Z}}\frac{1}{(m^2+n^2)^2}-\sideset{}{'}\sum_{\substack{m \text{ even}\\ n\text{ even}}}\frac{1}{(m^2+n^2)^2}=\frac{7}{16}\sideset{}{'}\sum_{m,n\in\mathbb{Z}}\frac{1}{(m^2+n^2)^2} = \sideset{}{'}\sum_{\substack{m \text{ even}\\ n\in\mathbb{Z}}}\frac{1}{(m^2+n^2)^2}.
\end{equation*}
Therefore, we have
\begin{equation}\label{E1}
\begin{aligned}
\sideset{}{'}\sum_{\substack{m \text{ even}\\ n\text{ even}}}\frac{1}{(m^2+n^2)^2}+\sideset{}{'}\sum_{\substack{m \text{ even}\\ n\in\mathbb{Z}}}\frac{1}{(m^2+n^2)^2} &= \frac{1}{2}\sideset{}{'}\sum_{m,n\in\mathbb{Z}}\frac{1}{(m^2+n^2)^2}\\
&= \frac{1}{2}\sideset{}{'}\sum_{m,n\in\mathbb{Z}}\frac{m^2+n^2}{(m^2+n^2)^3}\\
&= \sideset{}{'}\sum_{m,n\in\mathbb{Z}}\frac{m^2}{(m^2+n^2)^3}.
\end{aligned}
\end{equation}

Using \eqref{E1}, we then obtain 
\begin{equation}\label{E2}
\begin{aligned}
\sideset{}{'}\sum_{\substack{m \text{ even}\\ n\text{ even}}}\frac{1}{(m^2+n^2)^2}+\sideset{}{'}\sum_{\substack{m \text{ even}\\ n\in\mathbb{Z}}}\frac{n^2}{(m^2+n^2)^3} &= \sideset{}{'}\sum_{\substack{m \text{ even}\\ n\text{ even}}}\frac{1}{(m^2+n^2)^2}+\sideset{}{'}\sum_{\substack{m \text{ even}\\ n\in\mathbb{Z}}}\left(\frac{1}{(m^2+n^2)^2}-\frac{m^2}{(m^2+n^2)^3}\right)\\
&= \sideset{}{'}\sum_{m,n\in\mathbb{Z}}\frac{m^2}{(m^2+n^2)^3} - \sideset{}{'}\sum_{\substack{m \text{ even}\\ n\in\mathbb{Z}}}\frac{m^2}{(m^2+n^2)^3}\\
&=  \sideset{}{'}\sum_{\substack{m \text{ odd}\\ n\in\mathbb{Z}}}\frac{m^2}{(m^2+n^2)^3}.
\end{aligned}
\end{equation}

By \eqref{E:linear}, \eqref{E:L31} and \eqref{E2}, one sees that
\begin{align*}
\mathcal{L}^E_{3,1}((Q)+(P+Q)) &= \frac{1}{6\pi}\sideset{}{'}\sum_{m,n\in\mathbb{Z}}\frac{(-1)^m m^2}{(m^2+n^2)^3}\\
&= \frac{1}{6\pi}\sideset{}{'}\sum_{\substack{m \text{ even}\\ n\in\mathbb{Z}}}\frac{ m^2}{(m^2+n^2)^3}-\frac{1}{6\pi}\sideset{}{'}\sum_{\substack{m \text{ odd}\\ n\in\mathbb{Z}}}\frac{m^2}{(m^2+n^2)^3}\\
&=  \frac{1}{6\pi}\sideset{}{'}\sum_{\substack{m \text{ even}\\ n\in\mathbb{Z}}}\frac{ m^2}{(m^2+n^2)^3} -\frac{1}{6\pi}\sideset{}{'}\sum_{\substack{m \text{ even}\\ n\in\mathbb{Z}}}\frac{n^2}{(m^2+n^2)^3}-\frac{1}{96\pi}\sideset{}{'}\sum_{m,n\in\mathbb{Z}}\frac{1}{(m^2+n^2)^2}\\
&= -\frac{1}{3\pi}L(g,3) -\frac{\pi}{144}L(\chi_{-4},2),
\end{align*}
where the last equality follows from \eqref{E:m1} and \eqref{E:ch1}.

Finally, we prove \eqref{L5} using \eqref{E:L32}, \eqref{E:ch1}, \eqref{E:m1}, and some tedious manipulations.

\end{proof}

Theorem \ref{T:main} now easily follows from Lemma \ref{L:1} and Theorem \ref{T:CS}.
\begin{proof}[Proof of Theorem \ref{T:main}]
Let $\alpha = L(g,3)$ and $\beta = L(\chi_{-4},2)$. By \eqref{L1}-\eqref{L5} and Theorem \ref{T:CS}, we have 
\begin{align*}
\begin{vmatrix}
\mathcal{L}^E_{3,1}(\xi_1) & \mathcal{L}^E_{3,2}(\xi_1) \\ 
\mathcal{L}^E_{3,1}(\xi_2) & \mathcal{L}^E_{3,2}(\xi_2)
\end{vmatrix} &= \mathcal{L}^E_{3,1}(\xi_1)\mathcal{L}^E_{3,2}(\xi_2)-\mathcal{L}^E_{3,1}(\xi_2)\mathcal{L}^E_{3,2}(\xi_1) \\
&= \frac{16}{\pi}\alpha\left(-\frac{1}{3\pi}\alpha-\frac{43\pi}{48}\beta\right)+\frac{16}{3\pi}\alpha\left(\frac{1}{\pi}\alpha-\frac{43\pi}{32}\beta\right)\\
&= -\frac{43}{2}\alpha\beta \\
&= -\frac{43}{2}L(\Sym^2(E),3)
\end{align*}

The second equality in \eqref{E:main} follows from the functional equation for the symmetric square $L$-function.
\end{proof}


\bigskip

\noindent{\bf Acknowledgements} The original motivation for this note was to understand the possible relationship between Mahler measures of multi-variate polynomials and special $L$-values, which was the main theme of the author's Ph.D. thesis. The author would like to thank his adviser, Matt Papanikolas, for his encouragement and support. The author is also grateful to J\"{o}rn Steuding for pointing out possible extension of the main result to the congruent number elliptic curves.


\end{document}